\documentclass[11pt,a4paper]{article}
\usepackage{tikz}
\usepackage{subfigure}
\usepackage[english]{babel}
\usepackage{graphicx}
\usepackage{extarrows}
\usepackage{setspace}
\usepackage{color}
\usepackage{indentfirst, latexsym, bm, amssymb}

\begin{document}

\newtheorem{theorem}{Theorem}
\newtheorem{corollary}[theorem]{Corollary}
\newtheorem{definition}[theorem]{Definition}
\newtheorem{conjecture}[theorem]{Conjecture}
\newtheorem{question}[theorem]{Question}
\newtheorem{lemma}[theorem]{Lemma}
\newtheorem{proposition}[theorem]{Proposition}
\newtheorem{example}[theorem]{Example}
\newenvironment{proof}{\noindent {\bf
Proof.}}{\rule{3mm}{3mm}\par\medskip}
\newcommand{\remark}{\medskip\par\noindent {\bf Remark.~~}}
\newcommand{\pp}{{\it p.}}
\newcommand{\de}{\em}

\title{The anti-Ramsey number for paths}
\author{Long-Tu Yuan\thanks{School of Mathematical Sciences and Shanghai Key Laboratory of PMMP, East China Normal University, 500 Dongchuan Road, Shanghai 200240, P.R.  China.
Email: ltyuan@math.ecnu.edu.cn. Supported in part by National Natural Science Foundation of China grant 11901554 and Science and Technology Commission
of Shanghai Municipality (No. 18dz2271000, 19jc1420100).}
}
\date{}
\maketitle
\begin{abstract}
We determine the exactly anti-Ramsey number for paths. This confirms a conjecture posed by Erd\H{o}s, Simonovits and S\'{o}s in 1970s.
\end{abstract}


\vskip 0.5cm

\section{Introduction}
A subgraph of an edge-colored graph is {\it rainbow} if all of its edges have different colors. For a given graph $H$, the anti-Ramsey number AR$(n,H)$ of $H$ is the maximum number of colors  in an edge-colored $K_n$ such that $K_n$ does not contain a copy of rainbow $H$.

The anti-Ramsey number was introduced by Erd\H{o}s, Simonovits and S\'{o}s \cite{erdossim1973}.
In the same paper, they observed that if we color a copy of $K_{k-2}$ in $K_n$ with different colors and color the remaining edges a new color, then $K_n$ contains no rainbow paths on $k$ vertices.
Let $X\subseteq V(K_n)$ with size $\lfloor(k-3)/2\rfloor$ and let $i=1$ if $k$ is odd and $i=2$ if $k$ is even.
If we color the edges incident $X$ with different colors and color the remaining edges with $i$ new colors, then we can easily check that $K_n$ does not contain rainbow paths on $k$ vertices.
Hence, they asked whether those two configurations are best possible.

Denote by $P_k$ the path on $k$ vertices. Let $t=\lfloor(k-3)/2\rfloor$. In \cite{simsos1984anpath}, Simonovits and S\'{o}s determined AR$(n,P_k)$ for  $n\geq c_1t^2$, where $c_1$ is a constant.
They also claimed that their result held for $n\geq 5t/2+c_2$, where $c_2$ is a constant (without proof).
The exactly anti-Ramsey number for paths is still not know.
For other related results on this topic we refer the interested readers to a survey of Fujita, Magnant, and Ozeki \cite{Fujita2010} and some new results as \cite{meilu2019anti,Gorgol2016,Gu2020,Jahan2016,Jiang2009,Lan2019,Lu2020,Yuan2020,Yuan2020+}.

The main result of this paper is the following theorem which settles an old conjecture posed by Erd\H{o}s, Simonovits and S\'{o}s \cite{erdossim1973} almost fifty years ago.

\begin{theorem}{\bf \noindent.}\label{THM:main}
Let $P_k$ be a path on $k$ vertices and $\ell=\lfloor(k-1)/2\rfloor$. If $n\geq k\geq 5$,
then
\begin{equation*}
\emph{AR}(n,P_k)=\max\left\{{k-2 \choose 2}+1,{\ell-1 \choose 2}+(\ell-1)(n-\ell+1)+\epsilon\right\},
\end{equation*}
where $\epsilon=1$ if $k$ is odd and $\epsilon=2$ otherwise.
\end{theorem}

The extremal results of an $n$-vertex graph  proved by stability results usually need $n$ to be sufficiently large.
Our proof of Theorem~\ref{THM:main} bases on a recently    result of F\"{u}redi, Kostochka, Luo and Verstra\"{e}te \cite{Furedi2016,Furedi2018} holding for graphs with arbitrary number of vertices.
Hence, we can apply the stability results to determine the exactly anti-Ramsey number for paths.
It is very interesting to obtain other exactly extremal result by stability method.

\section{Notation and basic lemmas}
Given two graphs $G$ and $H$, we say that $G$ is {\it $H$-free} if $G$ does not contain a copy of $H$ as a subgraph.
For a given graph $H$, the Tur\'{a}n number  ex$(n,H)$ of $H$ is the maximum number of edges in an $n$-vertex $H$-free graph.
Similarly, the connected Tur\'{a}n number of a given graph $H$, denoted by ex$_{\mbox{con}}(n,H)$, is the maximum number of edges of an $n$-vertex $H$-free connected graph.
The anti-Ramsey problem is strongly connected with the Tur\'{a}n problem.
So we introduce some results about Tur\'{a}n problem first.

Erd\H{o}s and Gallai \cite{erdHos1959maximal} first studied the Tur\'{a}n numbers of paths. Later, Faudree and Schelp \cite{Faudree1975} and independently Kopylov \cite{Kopylov1977} improved Erd\H{o}s and Gallai's result to the following.

\begin{theorem}[Faudree, Schelp \cite{Faudree1975} and Kopylov \cite{Kopylov1977}]\label{THM:path}
Let $n\geq k$. Then $\emph{ex}(n,P_k)=s{k-1 \choose 2}+{r \choose 2}$, where $n=s(k-1)+r$ and $0\leq r\leq k-2$. Moreover, the extremal graphs are characterized.
\end{theorem}

For connected graphs without containing a copy of $P_k$, Balister, Gy\H{o}ri, Lehel and Schelp \cite{Balister2008} and independently Kopylov \cite{Kopylov1977} proved the following theorem.

We first introduce the  following graphs which play an important role in extremal problems for paths and cycles.
For integers $n\geq k\geq 2a$, let $H(n,k,a)$ be the $n$-vertex graph whose vertex set is partitioned into three sets $A, B, C $ such that $|A|=a, |B|=n-k+a$ and $|C|=k-2a$
and the edge set consists of all edges between $A$ and $B$ together with all edges in $A\cup C$.
Let $h(n,k,a)=e(H(n,k,a))$.

\begin{theorem}[Balister, Gy\H{o}ri, Lehel,  Schelp \cite{Balister2008} and Kopylov \cite{Kopylov1977}]\label{THM:connected-path}
Let $n\geq k$ and $s=\lfloor(k-2)/2\rfloor$. Then $\emph{ex}_{\emph{con}}(n,P_k)=\max\{h(n,k-1,1),h(n,k-1,s)\}$. Moreover, the extremal graph is either $H(n,k-1,1)$ or $H(n,k-1,s)$.
\end{theorem}

Let $n\geq k$ and $\ell=\lfloor(k-1)/2\rfloor$ be positive integers.
Define $$ar(n,k)=\max\{h(k,k-1,1)-1,h(n,k-1,\ell-1)-i\},$$ where $i=0$ if $k$ is odd and $i=1$ if $k$ is even.
Then Theorem~\ref{THM:main} states that if the number of colors in an edge-colored $K_n$ is at least ar$(n,k)+1$, then $K_n$ contains a rainbow copy of $P_k$.
The following lemma only needs Theorems~\ref{THM:path}, \ref{THM:connected-path} and some basic calculations. We move its proof to Appendix A.

\begin{lemma}\label{Lemma:L-is-nonconnected}
Let $k_1\geq k_2\geq 3$ and $t\geq 1$. Let $n_0\geq k_1-1$, $n_1\geq n_2\geq \ldots\geq n_t\geq 1$ and $n=\sum_{i=0}^{t}n_i\geq k_1+k_2-1\geq 5$. Then $$\emph{ex}_{\emph{con}}(n_0,P_{k_1})+\sum_{i=1}^{t}\emph{ex}_{\emph{con}}(n_i,P_{k_2})+t-1\leq ar(n,k_1+k_2-1).$$
\end{lemma}

Very recently, F\"{u}redi, Kostochka, Luo and Verstra\"{e}te \cite{Furedi2016,Furedi2018} considered the stability results of the well-known Erd\H{o}s-Gallai theorems on cycles and paths. Let $\mathcal{C}_{k}$ be the set of cycles of length at least $k$. Let $G$ be an $n$-vertex connected  $P_k$-free ($n$-vertex 2-connected  $\mathcal{C}_k$-free) graph. Their results  state that if $e(G)$ is close to the maximum value of number of edges of $n$-vertex connected  $P_k$-free (2-connected  $\mathcal{C}_k$-free) graphs, then $G$ must be a subgraph of some well-specified graphs. We will use the following two corollaries of the main theorems in \cite{Furedi2016,Furedi2018}  (Theorem 1.6 in \cite{Furedi2016} and Theorem 2.3\footnote{ Theorem 2.3 in \cite{Furedi2018} states the stability result for 2-connected $\mathcal{C}_k$-free graphs. Note that if we add new a vertex and join it to all vertices of a connected $P_k$-free graph, then the obtained graph is 2-connected $\mathcal{C}_k$-free. Theorem 2.3 in \cite{Furedi2018} can be extended to stability result for connected $P_k$-free graphs as Theorem 1.6 was extended by Theorem 1.4 in \cite{Furedi2016}.} in \cite{Furedi2018}, also see results in \cite{MN2019,MY20+}). We divide their results basing on the parity of $k$ for the purpose of proving our main result.

\begin{corollary}[F\"{u}redi,   Kostochka, Luo and Verstra\"{e}te \cite{Furedi2016,Furedi2018}]\label{Coro:1}
Let $k\geq 9$ be odd and $\ell=(k-1)/2$. Let $G$ be an $n$-vertex connected graph without containing a path on $k$ vertices.
Then $e(G)\leq\max\{h(n,k-1,2),h(n,k-1,\ell-1)\}$ unless $G$ is a subgraph of $H(n,k-1,1)$.
\end{corollary}


\begin{corollary}[F\"{u}redi,   Kostochka, Luo and Verstra\"{e}te \cite{Furedi2016,Furedi2018}]\label{Coro:2}
Let $k\geq 6$ be even and $\ell=\lfloor(k-1)/2\rfloor$. Let $G$ be an $n$-vertex connected graph without containing path on $k$ vertices.
Then $e(G)<\max\{h(n,k-1,2)+1,h(n,k-1,\ell-1)\}$ unless\\
(a) $G$ is a subgraph of $H(n,k-1,1)$, or\\
(b) $G$ is a subgraph of $H(n,k-1,\ell)$, or\\
(c) $G$ is an acyclic\footnote{We say a graph is {\it acyclic} if it is connected without containing a cycle.} $P_6$-free graph  when $k=6$, or \\
(d) $G=H(n,k-1,\ell-1)$.
\end{corollary}

\remark For $k=6$ in Corollary~\ref{Coro:2}, see Theorem 5.1(4) in \cite{Furedi2016}.
Actually, the result we used here is extended  by Theorem 5.1(4) in \cite{Furedi2016} as Theorem 1.6 was extended by Theorem 1.4 in \cite{Furedi2016}.
Corollary~\ref{Coro:2}$(d)$ ($e(G)=h(n,k-1,\ell-1)$) is not proved in \cite{Furedi2016,Furedi2018}.
We can prove this with a little more effort.
We refer the readers to  \cite{MY20+} for a short proof of stability results of the Erd\H{o}s-Gallai theorems from which one can easily get Corollary~\ref{Coro:2}$(d)$.

\medskip

We also need the following simple lemma\footnote{The lemma maybe appears in some old paper, but I do not find it.} (see Lemma 2.2 in \cite{Yuan2017}).
\begin{lemma}\label{P2l}
Let $G$ be a bipartite graph with classes $A$ and $B$. Let $|A|=|B|=\ell$.
If $e(G)\geq (\ell-1)\ell+2$, then $G$ contains a cycle of length $2\ell$.
\end{lemma}

\section{Proof of Theorem~\ref{THM:main}}

The {\it representing graph} of a graph $G$ with an edge coloring $c$ is a spanning subgraph of $G$ obtained by taking one edge of each color of $c$.
For a set of edges $E$ of $G$, we use $c(E)$ to denote the colors of edges in $E$.
For a set of colors $\mathcal{C}$, when an edge $e$ is colored by a color in $\mathcal{C}$, we say $e$ is colored by $\mathcal{C}$ for short.
Given a graph $G$, we use $T(G)$ to denote the set of its cut edges.

\begin{definition}
\emph{Given a graph $G$ with an edge coloring $c$, we say that the pair $(G,c)$ is a {\it good edge coloring} if there is a connected representing graph $L_n$ of $G$ with a non-empty set of cut edges $X\subseteq T(L_n)$ such that each  $e\in E(G)$ between components of $L_n-X$ are colored by $c(X)$.}
\end{definition}

\begin{lemma}\label{Lem:good coloring}
Let $G$ be a graph with an edge coloring $c$ and let $\mathcal{L}_n$ be the set of connected representing graphs of $c$.
If   $\mathcal{C}_0=\bigcap_{L_n \in \mathcal{L}_n}\{c(e):e\in T(L_n)\}$ is not empty, then $(G,c)$ is a good edge coloring with a representing graph $L^\ast_n$ and a set of cut edges $X\subseteq T(L^\ast_n)$ such that $\mathcal{C}_0\subseteq c(X)$.
\end{lemma}
\begin{proof}
Let $L^0_n\in \mathcal{L}_n$ be a representing graph of $G$.
Since $\mathcal{C}_0=\bigcap_{L_n \in \mathcal{L}_n}\{c(e):e\in T(L_n)\}$ is not empty, $L^0_n$ is connected with at least one cut edge.
We obtain a subgraph of $L^0_n$ by the following procedure.
Delete the edges of $L_n$ colored by $\mathcal{C}_0$ and denote the obtained graph by $L^1_n$.
Let $\mathcal{C}_1$ be the colors of the edges of $G$ between any two components of $L^1_n$.
Delete the edges of $L_n$ colored by $\mathcal{C}_1$ and denote the obtained graph by $L^2_n$.
We go on this procedure and finally obtain a minimal spanning subgraph $L^t_n$ of $L_n$ for some integer $t$.

It is enough to show that the edges of $G$ between any two components of $L^t_n$ are colored by $c(X)$ with $X \subseteq T(L^0_n)$.
We will show that $L_n^i$ is obtained from $L_n^0$ by deleting edges of $T(L^0_n)$   for $i=1,\ldots,t$.
This will complete our proof of the lemma.
Suppose for contrary that there is an edge $e^\prime_{\ell_0}\in E(G)$ between two components of $L^{\ell_0}_n$ not colored by $c(T(L^0_n))$.
We choose $\ell_0$ as small as possible.
Clearly, we have $\ell_0>0$ and each $L_n^i$ for $i<\ell_0$ is obtained from $L_n^0$ by deleting edges from $T(L^0_n)$.
Then we may add $e_{\ell_0}^\prime$ to $L_n^0$ and delete the edge $e_{\ell_0}$ of $L^0_n$ colored by $c(e_{\ell_0}^\prime)$.
Since $e_{\ell_0}$ is not a cut edge of $L_n^0$, the obtained graph $\widetilde{L}^1_n\in\mathcal{L}_n$ is connected.
Moreover, by the minimality of $\ell_0$, $\widetilde{L}^1_n$ contains a cycle containing an edge $e_{\ell_1}$ of $T(L_n)$ colored by $\mathcal{C}_{\ell_1}$  with $0\leq \ell_1\leq \ell_0-1$.
We choose $\ell_1$ as small as possible.
If $\ell_1=0$, then $e_{\ell_1}$ is colored by $\mathcal{C}_0$ and $e_{\ell_1}$ is not a cut edge of $\widetilde{L}^1_n$.
This is a contradiction to definition of $\mathcal{C}_0$.
Let $\ell_1\geq 1$.
Then, from $\widetilde{L}_n^1$, after deleting the edge $e_{\ell_1}$ and adding an edge $e^\prime_{\ell_1}$ colored by $c(e_{\ell_1})$ between the components of $L_n^{\ell_1}$, the obtained graph $\widetilde{L}^2_n\in\mathcal{L}_n$ is connected.
Moreover, $\widetilde{L}^2_n$ contains a cycle   containing one edge $e_{\ell_2}$ of $T(L_n)$ colored by $\mathcal{C}_{\ell_2}$ with $\ell_2<\ell_1$.
If $\ell_2=0$, then get a contradiction to definition of $\mathcal{C}_0$.
Otherwise, we may go on the above procedure until we have $\ell_{s}=0$ for some $s$.
Thus there is an edge $e_{\ell_s}\in T(L_n)$ colored by $\mathcal{C}_0$ and this edge is not a cut edge for $\widetilde{L}^s_n\in\mathcal{L}_n$.
This final contradiction completes the proof of Lemma~\ref{Lem:good coloring}.
\end{proof}


Let $c$ be an edge coloring of $K_n$ with maximum number of colors such that $K_n$ contains no copy of rainbow $P_k$.
Taking a representing graph $L_n$ of $K_n$ with a maximum component.
The configurations before Theorem~\ref{THM:main} show that $e(L_n)\geq ar(n,k)$.
Suppose that
\begin{equation}\label{eq-for-main-theorem-1}
e(L_n)\geq ar(n,k)+1.
\end{equation}
We will finish our proof of Theorem~\ref{THM:main} by contradictions.

\medskip

{\bf Claim.} $L_n$ is connected.

\medskip

\begin{proof}
Suppose for contrary that $L_n$ is not connected.
Let $C_1,C_2,\ldots, C_{\ell}$ be the components of $L_n$ with  $|C_1|\geq |C_2|\geq \ldots\geq |C_{\ell}|$.
We choose $L_n$ as following.
For $i=1,\ldots,\ell$, subject to the choice of $C_i$, we choose $C_{i+1}$ with $|V(C_{i+1})|=s_{i+1}$ maximum first and then $|E(C_{i+1})|=t_{i+1}$ maximum.

Let $T(C_1)$ be the set of cut edges of $C_1$.
Then $T(C_1)$ is not empty.
Otherwise, $C_1$ is 2-connected.
Let $X=V(C_1)$ and  $Y=V(L_n)\setminus V(C_1)$.
Choose an edge $xy$ in $K_n$ with $x\in X$ and $y\in Y$.
Let $L^\prime_n$ be the representing graph obtained from $L_n$ by adding the edge $xy$ and deleting the edge in $L_n$ with color $c(xy)$.
Then $L^\prime_n$ contains a component with size $s_1+1$, a contradiction to our choice of $L_n$.
Moreover, each edge between $X$ and $Y$ are colored by $T(C_1)$.

Let $\mathcal{L}^\ast_n$ be the set of representing graphs of $K_n$ containing a component with $s_1$ vertices and $t_1$ edges. Let
$$\mathcal{C}^\ast=\bigcap_{L^\prime_n \in \mathcal{L}^\ast_n}\{c(e):e\in T(L^\prime_n[C_1])\}.$$

Then each edge of $K_n$ between $X$ and $Y$ is colored by a color in $\mathcal{C}^\ast$.
Otherwise, as the above argument (for some $L^\prime_n \in \mathcal{L}^\ast_n$), there is a representing graph  containing a component with size $s_1+1$, a contradiction.
In particular, we show that  $\mathcal{C}^\ast$ is not empty.
Hence, it follows from Lemma~\ref{Lem:good coloring} that the pair $(c,K_n[C_1])$ is a good coloring with a representing graph $L^1_n[C_1]$ and a set of cut edges $X_1\subseteq T(L^1_n[C_1])$ such that $\mathcal{C}^\ast\subseteq c(X_1)$.
Let $P^1$ be a longest path in  $L^1_n[C_1]-X_1$ on $k_1$ vertices.
Without loss of generality, let $\widetilde{C}_1$ be the component of $L^1_n[C_1]-X_1$ containing a copy of $P_{k_1}$.
Assume  $L^1_n-X_1-V(\widetilde{C}_1)$ contains a path $\widetilde{P}^1$ on $k-k_1$ vertices.
Since each edge between $P^1$ and $\widetilde{P}^1$ are  colored by $c(X_1)$, $K_n$ contains a rainbow copy of $P_k$, a contradiction.
Thus $L^1_n-X_1-V(\widetilde{C}_1)$ is $P_{k-k_1}$-free.
If $(k-1)/2\leq k_1\leq k-3$, then since the number of components of $L^1_n-X_1$  is at least  $|c(X_1)|+2$, by Lemma~\ref{Lemma:L-is-nonconnected}, we have $e(L^1_n)\leq ar(n,k)$, a contradiction to (\ref{eq-for-main-theorem-1}).
If $k_1=k-2$, then $L^1_n-X_1-V(\widetilde{C}_1)$ is an independent set.
In this case we can consider  $K_n[C_1]$ since it contains $ar(n,k)+1$ colors.
Suppose that $k_1< (k-1)/2$.
Then we consider the subgraph $G_2$  of $G_1=K_n$ obtained by deleting $V(C_1)$ and the edges colored by $c(E(C_1))$.
Then as the previous argument, by Lemma~\ref{Lem:good coloring}, the pair $(c,G_2[C_2])$ is a good coloring  with a representing graph $L^2_n[C_2]$ and a set of cut edges $X_2\subseteq T(L^2_n[C_2])$.
Let $\widetilde{C}_2$ be the component of $L^2_n[C_2]-X_2$  containing a longest path $P^2$ on $k_2$ vertices.
If $(k-1)/2\leq k_2\leq k-3$,  then by Lemma~\ref{Lemma:L-is-nonconnected}, we have $e(L^2_n)\leq ar(n,k)$, a contradiction (other components of $L^2_n-X_1-X_2$ do not contain a path on $k-k_2$ vertices).
If $k_1=k-2$, then $L^1_n-X_1-X_2-V(\widetilde{C}_2)$ is an independent set.
Hence we can consider $K_n[C_1\cup C_2]$ since $K_n[C_1\cup C_2]$ contains $ar(n,k)+1$ colors.
Suppose that $k_2< (k-1)/2$.
We may go on this procedure and obtain $X_{\ell}$ such that $L^\ell_n[C_\ell]-X_{\ell}$ ($X_\ell$ can be an empty set) contains a longest path $P^{\ell}$ on $k_\ell<(k-1)/2$ vertices or $k-2$ vertices.
If $P^\ell$ has less than $(k-1)/2$ vertices then each component of $L_n-\bigcup_{i=1}^{\ell}X_i$ dose not contain a path on at least $(k-1)/2$ vertices, where $L_n=\bigcup_{i=1}^\ell L_n^i[C_i]$.
Moreover, the number of components of $L_n-\bigcup_{i=1}^{\ell}X_i$  is  $\sum_{i=1}^{\ell}(c(X_i)+1)$.
Thus, since $\lceil(k-1)/2\rceil-1\leq k-3$,  by Lemma~\ref{Lemma:L-is-nonconnected}, we have $e(L_n)\leq ar(n,k)$, a contradiction to (\ref{eq-for-main-theorem-1}).
If $P^\ell$ has $k-2$ vertices, then each of $C_i$ is a tree for $1\leq i\leq \ell-1$.
Let $x_1x_2$ be a pendent edge of $C_1$, where $x_1$ is a leaf of $C_1$.
If there is an edge color by $c(x_1x_2)$ between $V(C_1)\setminus\{x_1\}$ and $L_n-V(C_1)$, then $K_n-\{x_1\}$ contains $ar(n,k)+1$ colors.
Thus we may consider the graph $K_n-\{x_1\}$.
Hence the edges between $V(C_1)\setminus\{x_1\}$ and $L_n-V(C_1)$ are not colored by $c(x_1x_2)$.
In particular, the edges between $x_2$ and $L_n-V(C_1)$ are not colored by $c(x_1x_2)$.
Thus there is a path on $k$ vertices containing the edge $x_1x_2$, a contradiction.
The proof of the claim is complete.
\end{proof}

By the claim, $L_n$ is a connected graph. We divide the proof basing on the parity of $k$.

\medskip

{\bf Case 1. $k$ is odd, i.e., $k=2\ell+1$.}

\medskip
For $k=5,7$, we have $ar(n,k)+1>\mbox{ex}_{\mbox{con}}(n,P_k)$.
Hence $L_n$ contains a copy of $P_k$, a contradiction.
Let $k\geq 9$, i.e., $\ell\geq 4$.
Note that $ar(n,k)+1> \max\{h(n,k-1,2),h(n,k-1,\ell-1)\}$.
By Corollary~\ref{Coro:1} and (\ref{eq-for-main-theorem-1}), each connected representing subgraph of $K_n$ is  a subgraph of $H(n,k-1,1)$.

Let $A\cup B \cup C$ be the partition of $H(n,k-1,1)$ in definition.
Thus, since $L_n$ is connected, each vertex in $B$ has degree one in $L_n$.
Basic calculation shows that if $n>(5\ell-1)/2> (5\ell^2-11\ell)/(2\ell-4) $, then $ar(n,k)+1> h (n,k-1,1)$, a contradiction.
Hence, we can assume $n\leq(5\ell-1)/2$.
By (\ref{eq-for-main-theorem-1}), there are at most $n-2\ell-1\leq (5\ell-1)/2-2\ell<2\ell-4$ non-edges of $L_n$ inside  $A\cup C$.
Let $L^\ast_n$ be the representing graph obtained from $L_n$ by adding an edge $e$ inside $B$ and deleting the edge of $L_n$ colored by $c(e)$.
Thus $L^\ast_n$ contains at most $n-k+1$ with degree one and hence $L^\ast_n$  is not a subgraph of $H(n,k-1,1)$, a contradiction when $L^\ast_n$ is connected.
Suppose that $L^\ast_n$ is not connected.
Then $L^\ast_n$ contains a unique isolated vertex, say $x$.
Let $L^\ast_{n-1}=L_n^\ast-\{x\}.$
Then we have $e(L^\ast_{n-1})\geq ar(n,k)+1\geq ar(n-1,k)+1$.
Go on the previous arguments repeatedly we can finally get a contradiction.
The proof of Case 1 is complete.

\medskip

{\bf Case 2. $k$ is even, i.e., $k=2\ell+2$.}

\medskip

Note that $ar(n,k)+1\geq\max\{h(n,k-1,2)+1,h(n,k-1,\ell-1)\}$.
By Corollary~\ref{Coro:2} and (\ref{eq-for-main-theorem-1}), for each  connected representing graph $L_n$ of $K_n$, we have the following:
\begin{itemize}
  \item $(a)$ $L_n$ is a subgraph of $H(n,k-1,1)$, or
  \item $(b)$ $L_n$ is a subgraph of $H(n,k-1,\ell)$, or
  \item $(c)$ $L_n$ is an acyclic $P_6$-free graph  when $k=6$, or
  \item $(d)$ $L_n=H(n,k-1,\ell-1)$.
\end{itemize}

For $(b)$, let $A\cup B \cup C$ be the partition of $H(n,k-1,\ell)$ in definition, i.e, $|A|=\ell$, $|C|=1$ and $|B|=n-\ell-1$.
Let $X=A$ and $Y=B\cup C$.
By (\ref{eq-for-main-theorem-1}) there are at least $(\ell-1) |B\cup C|+3$ edges between $X$ and $Y$.
Hence, there exists a vertex $y$ in $Y$ with degree $\ell$.
Moreover, since  $L_n[X,Y\setminus\{y\}]$ contains a subgraph on $2\ell$ vertices containing $X$ with  $(\ell-1)\ell+2$ edges, by Lemma~\ref{P2l}, $L_n[X,Y\setminus\{y\}]$ contains a cycle, $\widetilde{C}$, of length $2\ell$.
Let $yy^\prime$ be an edge in $K_n[Y]$ with $y^\prime\notin V(\widetilde{C})$.
Then one can find a rainbow path of length $2\ell+2$ in $K_n[V(\widetilde{C})\cup \{y,y^\prime\}]$, a contradiction.
For $(c)$, we have $e(L_n)\leq n-1$, contradicts (\ref{eq-for-main-theorem-1}).
For $(d)$, let $L_n=H(n,k-1,\ell-1)$.
Let $L^\ast_n$ be representing graph obtained from $L_n$ by adding an edge $e$ not in $L_n$ and deleting the edge colored by $c(e)$ in $L_n$.
It is obviously that $L_n^\ast$ contains a copy of $P_k$, a contradiction.
Finally, let $L_n$ be a subgraph of $H(n,k-1,1)$.
Then we have $e(L_n)\leq h(n,k-1,1)$.
Combining with (\ref{eq-for-main-theorem-1}), we have $n\leq \lfloor(5\ell+3)/2\rfloor$ for $\ell\geq 3$.
If $k\geq 8$, i.e., $\ell\geq 3$, then by (\ref{eq-for-main-theorem-1}), there are at most $n-2\ell-2\leq (5\ell+3)/2-2\ell\leq2\ell-3$ non-edges of $L_n$ inside  $A\cup C$.
We get a contradiction similarly as Case 1.
If $k=6$, then by (\ref{eq-for-main-theorem-1}) we have  $L_n=H(n,5,1)$.
Hence, it is easy to check that $K_n$ contains a rainbow copy of $P_6$.
This final contradiction completes our proof of Theorem~\ref{THM:main}.

\appendix
\section{Proof of Lemma~\ref{Lemma:L-is-nonconnected}}

\begin{proof}
The proof of Lemma~\ref{Lemma:L-is-nonconnected} bases on Theorems~\ref{THM:path} and \ref{THM:connected-path} and some basic calculations.
Let $k_1\geq k_2\geq 3$ and $t\geq 1$. Let $n_0\geq k_1-1$, $n_1\geq n_2\geq \ldots\geq n_t\geq 1$ and $n=\sum_{i=0}^{t}n_i\geq k_1+k_2-1\geq 5$.
Since $\mbox{ex}_{\mbox{con}}(n_i,P_{k_1})\leq n_i-1$ for $k_1\leq 4$ and $n_i\geq k_1$, the lemma holds easily for $k_1\leq 4$.
Hence, we may suppose that $k_1\geq  5$.

\medskip

{\bf Claim.} Let $m=s(k-1)+r$ with $s\geq 0$, $k\geq 2$ and $1\leq r\leq k-1$. Let $m_1\geq m_2\geq \ldots \geq m_t\geq 1$ and  $m=\sum_{i=1}^{t}m_i$. Then $\sum_{i=1}^{t}\mbox{ex}(m_i,P_k)+t-1\leq \mbox{ex}(m,P_k)+s.$

\medskip

\begin{proof} Clearly, we have $\sum_{i=1}^{t}\mbox{ex}(m_i,P_k)\leq \mbox{ex}(m,P_k).$
Thus the claim holds trivially for $t-1\leq s$.
Now suppose that $t-1>s$.
Then we have $1\leq m_t\leq m_{t-1}\leq k-2.$
Note that $\mbox{ex}(n,P_k)={ n \choose 2}$ for $n\leq k-1$.
By Theorem~\ref{THM:path},
we have $\mbox{ex}(m_{t-1},P_k)+\mbox{ex}(m_{t},P_k)+1={m_{t-1} \choose 2}+{m_t \choose 2}+1\leq \mbox{ex}(m_{t-1}+m_t,P_k)$.
Thus we have
$$\sum_{i=1}^{t}\mbox{ex}(m_i,P_k)+t-1\leq \sum_{i=1}^{t-2}\mbox{ex}(m_i,P_k)+\mbox{ex}(m_{t-1}+m_t,P_k)+t-2.$$
If $t-2\leq s$, then we are done.
Suppose that $t-2>s$.
Repeating the above argument $t-s-2$ times (reorder $m_1,m_2,\ldots,m_{t-2},m_{t-1}+m_t$), we have
$$\sum_{i=1}^{t}\mbox{ex}(m_i,P_k)+t-1\leq \sum_{i=1}^{s}\mbox{ex}(m^\prime_i,P_k)+\mbox{ex}\left(\sum_{i=s+1}^{s+2}m^\prime_{i},P_k\right)+s\leq \mbox{ex}(m,P_k)+s,$$
where $1\leq m^\prime_{s+1}\leq m^\prime_{s+2}\leq k-2$ and   $\sum_{i=1}^{s+2} m^\prime_i=m$.
The proof of the claim is complete.\end{proof}

Let $n-n_0=s^\prime(k_2-1)+r^\prime$ with $1\leq r^\prime\leq k_2-1$. By the claim, we have
$$\mbox{ex}_{\mbox{con}}(n_0,P_{k_1})+\sum_{i=1}^{t}\mbox{ex}_{\mbox{con}}(n_i,P_{k_2})+t-1\leq \mbox{ex}_{\mbox{con}}(n_0,P_{k_1})+\mbox{ex}(n-n_0,P_{k_2})+s^\prime.$$
Let $s_1=\lfloor(k_1-2)/2\rfloor$.
Since $k_2\geq 3$, we have $s_1\leq \ell-1$.
We will finish our proof in  the following two cases.

\medskip

{\bf Case 1.} $k_1+k_2-1$ is odd.

\medskip

Let $k_1+k_2-1=k=2\ell+1$. Basic calculation shows that $ar(n,k)=h(k,k-1,1)-1$ for $n\leq(5\ell-2)/2 $ and $ar(n,k)=h(n,k-1,\ell-1)$ for $n\geq (5\ell-2)/2$.
Let $n\leq\lfloor(5\ell-2)/2\rfloor$. We divide the proof into the following three subcases:

(a.1) $\mbox{ex}_{\mbox{con}}(n_0,P_{k_1})={k_1-1 \choose 2}$, i.e, $n_0=k_1-1$.
By Theorem~\ref{THM:path} and a detailed calculation, we have
\begin{align*}
&\mbox{ex}_{\mbox{con}}(n_0,P_{k_1})+\mbox{ex}(n-n_0,P_{k_2})+s^\prime={k_1-1 \choose 2}+ \mbox{ex}(n-k_1+1,P_{k_2})+s^\prime\\
\leq &{k_1-1 \choose 2}+ \mbox{ex}\left(\left\lfloor\frac{5\ell-2}{2}\right\rfloor-k_1+1,P_{k_2}\right)+s^\prime\leq {k_1+k_2-3 \choose 2}+1=h(k,k-1,1)-1.
\end{align*}

(a.2) $\mbox{ex}_{\mbox{con}}(n_0,P_{k_1})=h(n_0,k_1,1)$. By $k_1\geq 5$, calculations in \cite{Balister2008} show that if $k_1$ is even, then $k_1\leq n_0\leq   (5k_1-10)/4$ and if $k_1$ is odd, then $k_1\leq n_0\leq   (5k_1-7)/4$. Hence, we have
\begin{align*}
&\mbox{ex}_{\mbox{con}}(n_0,P_{k_1})+\mbox{ex}(n-n_0,P_{k_2})+s^\prime={k_1-2 \choose 2}+(n_0-k_1+2)+ \mbox{ex}(n-n_0,P_{k_2})+s^\prime\\
<&{k_1-1 \choose 2}+ \mbox{ex}(n-k_1+1,P_{k_2})+s^\prime<  {k_1+k_2-3 \choose 2}+1=h(k,k-1,1)-1,
\end{align*}
where the last strict inequality holds similarly as before.

(a.3) $\mbox{ex}_{\mbox{con}}(n_0,P_{k_1})=h(n_0,k_1,s_1)$, i.e, $n_0\geq   (5k_1-10)/4$ for even $k_1$ and $n_0\geq   (5k_1-7)/4$ for odd $k_1$.
Let $i_1=1$ when $k_1$ is odd and $i_1=0$ when $k_1$ is even. Then
\begin{align*}
&\mbox{ex}_{\mbox{con}}(n_0,P_{k_1})+\mbox{ex}(n-n_0,P_{k_2})+s^\prime\\
\leq&{s_1 \choose 2}+s_1(n_0-s_1)+i_1+ \mbox{ex}\left(\frac{5\ell-2}{2}-n_0,P_{k_2}\right)+s^\prime\\
<& h\left(\frac{5\ell-2}{2},k-1,\ell-1\right)=h(k,k-1,1)-1,
\end{align*}
where the strict inequality holds by Theorem~\ref{THM:path} and a detailed calculation.
Thus the lemma holds for $n\leq(5\ell-2)/2$ in Case 1.

Now we may assume that $n\geq\lceil(5\ell-2)/2\rceil$.
Note that $\mbox{ex}(n^\prime,P_{k_2})+s_1^\prime-\mbox{ex}(n^\prime-1,P_{k_2})-s^\prime_2\leq k_2-2\leq \ell-1$, where $s^\prime_1=\lceil n^\prime/(k_2-1)\rceil-1$ and $s^\prime_2=\lceil (n^\prime-1)/(k_2-1)\rceil-1$.
We divide the proof into the following three subcases:

(b.1) $\mbox{ex}_{\mbox{con}}(n_0,P_{k_1})={k_1-1 \choose 2}$, i.e, $n_0=k_1-1$.
Then
\begin{align*}
&\mbox{ex}_{\mbox{con}}(n_0,P_{k_1})+\mbox{ex}(n-n_0,P_{k_2})+s^\prime\\
= &{k_1-1 \choose 2}+s^\prime{k_2-1 \choose 2}+{r^\prime \choose 2}+s^\prime\\
 < &h(n,k-1,\ell-1),
\end{align*}
where the strict inequality holds by $n\geq\lceil(5\ell-2)/2\rceil$.

(b.2) $\mbox{ex}_{\mbox{con}}(n_0,P_{k_1})=h(n_0,k_1,1)$, i.e, $n_0\leq   (5k_1-10)/4$ for even $k_1$ and $n_0\leq   (5k_1-7)/4$ for odd $k_1$.
Then
\begin{align*}
&\mbox{ex}_{\mbox{con}}(n_0,P_{k_1})+\mbox{ex}(n-n_0,P_{k_2})+s^\prime\\
=& {k_1-2 \choose 2}+(n_0-k_1+2)+ s^\prime{k_2-1 \choose 2}+{r^\prime \choose 2}+s^\prime\\
< &h(n,k-1,\ell-1),
\end{align*}
where the strict inequality holds by $n\geq\lceil(5\ell-2)/2\rceil$.

(b.3) $\mbox{ex}_{\mbox{con}}(n_0,P_{k_1})=h(n_0,k_1,s_1)$, i.e, $n_0\geq   (5k_1-10)/4$ for even $k_1$ and $n_0\geq   (5k_1-7)/4$ for odd $k_1$.
Let $i_1=1$ when $k_1$ is odd and $i_1=0$ when $k_1$ is even.
Recall that $s_1\leq \ell-1$, we  have
\begin{align*}
&\mbox{ex}_{\mbox{con}}(n_0,P_{k_1})+\mbox{ex}(n-n_0,P_{k_2})+s^\prime\\
=& {s_1 \choose 2}+s_1(n-s_1)+i_1+s^\prime{k_2-1 \choose 2}+{r^\prime \choose 2}+s^\prime\\
< & h(n,k-1,\ell-1),
\end{align*}
where the strict inequality holds by $n\geq\lceil(5\ell-2)/2\rceil$.
We finish the proof of the lemma for Case 1.

\medskip

{\bf Case 2.} $k_1+k_2-1$ is even.

\medskip

Let $k_1+k_2-1=k=2\ell+2$. Basic calculation shows that $a(n,k)=h(k,k-1,1)-1$ for $n\leq(5\ell+2)/2 $ and $a(n,k)=h(n,k-1,\ell-1)-1$ for $n\geq (5\ell+2)/2$.
Let $n\leq\lfloor(5\ell+2)/2\rfloor$.
We divide the proof into the following three subcases:

(a.1) $\mbox{ex}_{\mbox{con}}(n_0,P_{k_1})={k_1-1 \choose 2}$, i.e, $n_0=k_1-1$.
Then we have
\begin{align*}
&\mbox{ex}_{\mbox{con}}(n_0,P_{k_1})+\mbox{ex}(n-n_0,P_{k_2})+s^\prime={k_1-1 \choose 2}+ \mbox{ex}(n-k_1+1,P_{k_2})+s^\prime\\
\leq &{k_1-1 \choose 2}+ \mbox{ex}\left(\left\lfloor\frac{5\ell+2}{2}\right\rfloor-k_1+1,P_{k_2}\right)+s^\prime< {k_1+k_2-3 \choose 2}+1=h(k,k-1,1)-1.
\end{align*}

(a.2) $\mbox{ex}_{\mbox{con}}(n_0,P_{k_1})=h(n_0,k_1,1)$, i.e, $n_0\leq   (5k_1-10)/4$ for even $k_1$ and $n_0\leq   (5k_1-7)/4$ for odd $k_1$.
Then
\begin{align*}
&\mbox{ex}_{\mbox{con}}(n_0,P_{k_1})+\mbox{ex}(n-n_0,P_{k_2})+s^\prime={k_1-2 \choose 2}+(n_0-k_1+2)+ \mbox{ex}(n-n_0,P_{k_2})+s^\prime\\
< & {k_1+k_2-3 \choose 2}+1=h(k,k-1,1)-1.
\end{align*}

(a.3) $\mbox{ex}_{\mbox{con}}(n_0,P_{k_1})=h(n_0,k_1,s_1)$, i.e, $n_0\geq   (5k_1-10)/4$ for even $k_1$ and $n_0\geq   (5k_1-7)/4$ for odd $k_1$.
Let $i_1=1$ when $k_1$ is odd and $i_1=0$ when $k_1$ is even. Then
\begin{align*}
&\mbox{ex}_{\mbox{con}}(n_0,P_{k_1})+\mbox{ex}(n-n_0,P_{k_2})+s^\prime\\
\leq &{s_1 \choose 2}+s_1(n_0-s_1)+i_1+ \mbox{ex}\left(\left\lfloor\frac{5\ell+2}{2}\right\rfloor-n_0,P_{k_2}\right)+s^\prime\\
< & h\left(\left\lfloor\frac{5\ell+2}{2}\right\rfloor,k-1,\ell-1\right)=h(k,k-1,1)-1.
\end{align*}

Let $n\geq\lceil(5\ell+2)/2\rceil$.
Recall that $\mbox{ex}(n^\prime,P_{k_2})+s_1^\prime-\mbox{ex}(n^\prime-1,P_{k_2})-s^\prime_2\leq k_2-2\leq \ell-1$.
We divide the proof into the following three subcases:

(b.1) $\mbox{ex}_{\mbox{con}}(n_0,P_{k_1})={k_1-1 \choose 2}$, i.e, $n_0=k_1-1$.
Then
\begin{align*}
&\mbox{ex}_{\mbox{con}}(n_0,P_{k_1})+\mbox{ex}(n-n_0,P_{k_2})+s^\prime={k_1-1 \choose 2}+ \mbox{ex}(n-k_1+1,P_{k_2})+s^\prime\\
=& {k_1-1 \choose 2}+s^\prime{k_2-1 \choose 2}+{r^\prime \choose 2}+s^\prime <  h(n,k-1,\ell-1)-1,
\end{align*}
where the strict inequality holds by $n\geq\lceil(5\ell+2)/2\rceil$.

(b.2) $\mbox{ex}_{\mbox{con}}(n_0,P_{k_1})=h(n_0,k_1,1)$, i.e, $n_0\leq   (5k_1-10)/4$ for even $k_1$ and $n_0\leq   (5k_1-7)/4$ for odd $k_1$.
Then
\begin{align*}
&\mbox{ex}_{\mbox{con}}(n_0,P_{k_1})+\mbox{ex}(n-n_0,P_{k_2})+s^\prime={k_1-2 \choose 2}+(n_0-k_1+2)+ \mbox{ex}(n-n_0,P_{k_2})+s^\prime\\
< & {\ell-1 \choose 2}+(\ell-1)(n-\ell+1)+1=h(n,k-1,\ell-1)-1,
\end{align*}
where the strict inequality holds by $n\geq\lceil(5\ell+2)/2\rceil$.

(b.3) $\mbox{ex}_{\mbox{con}}(n_0,P_{k_1})=h(n_0,k_1,s_1)$, i.e, $n_0\geq   (5k_1-10)/4$ for even $k_1$ and $n_0\geq   (5k_1-7)/4$ for odd $k_1$.
Let $i_1=1$ when $k_1$ is odd and $i_1=0$ when $k_1$ is even.
Recall that $s_1\leq \ell-1$, we have
\begin{align*}
&\mbox{ex}_{\mbox{con}}(n_0,P_{k_1})+\mbox{ex}(n-n_0,P_{k_2})+s^\prime={s_1 \choose 2}+s_1(n_0-s_1)+i_1+ \mbox{ex}(n-n_0,P_{k_2})+s^\prime\\
< & {\ell-1 \choose 2}+(\ell-1)(n-\ell+1)+1=h(n,k-1,\ell-1)-1,
\end{align*}
where the strict inequality holds by $n\geq\lceil(5\ell+2)/2\rceil$.
The proof is thus complete.
\end{proof}
\end{document}